\documentclass[12pt]{article}
\usepackage{amssymb}
\usepackage{amsmath,amsthm}
\usepackage[latin1]{inputenc}
\usepackage{color}
\usepackage{graphicx}
\usepackage{tikz}

\usepackage{epsfig}

 \newtheorem{remark}{Remark}

 \newtheorem{theorem}[remark]{Theorem}
 
 \newtheorem{corollary}[remark]{Corollary}

\title{The Hosoya polynomial of distance-regular graphs}

\author{Emeric Deutsch\\
Polytechnic Institute of New York University,
United States\\
emericdeutsch\@@msn.com
\\
\\
Juan A. Rodr\'{\i}guez-Vel\'{a}zquez
\\
 Departament d'Enginyeria Inform\`atica i Matem\`atiques
\\
 Universitat Rovira i Virgili\\
   Av. Pa\"{\i}sos Catalans 26,
43007 Tarragona, Spain.
\\ juanalberto.rodriguez\@@urv.cat}

\date{}

\begin{document}
\maketitle

\begin{abstract}
In this note we obtain an explicit formula for the Hosoya polynomial of any distance-regular  graph in terms of its intersection array. As a consequence, we obtain a very simple formula for the Hosoya polynomial of any strongly regular graph.
\end{abstract}

\section{Introduction}

Throughout this paper  $G=(V,E)$  denotes a connected,
simple and  finite  graph  with vertex set $V=V(G)$  and  edge set $E=E(G)$.    

The {\em distance} $d(u,v)$  between two vertices $u$ and
$v$ is the minimum  of the lengths of  paths between  $u$ and $v$.
The  {\em diameter} $D$ of  a graph $G$ is
defined as $$D:=\max_{u,v\in V(G)}\{d (u,v)\}.$$

The {\em Wiener index} $W(G)$ of a graph $G$ with vertex
set $\{v_1,v_2,...,v_n\}$,  defined as the sum of distances between
all pairs of vertices of $G$,
 $$W(G):= \frac{1}{2}\sum_{i=1,j=1}^nd(v_i,v_j),$$
 is the first mathematical invariant reflecting the
topological structure of a molecular graph.

This topological index  has been extensively studied; for
instance, a comprehensive survey on the direct calculation,
applications, and the relation of the Wiener index of trees with
other parameters of graphs can be found in \cite{Dobrynin2001}.
Moreover, a list of 120 references of the main works on the Wiener
index of graphs can be found in the referred survey.

The Hosoya polynomial of a graph was introduced in the Hosoya's seminal paper \cite{Hosoya1988239}
in 1988 and received a lot of attention afterwards. The polynomial was later
independently introduced and considered by Sagan, Yeh, and Zhang \cite{BruceEtAll1996} under the
name Wiener polynomial of a graph. Both names are still used for the polynomial
but the term Hosoya polynomial is nowadays used by the majority of researchers.
The main advantage of the Hosoya polynomial is that it contains a wealth of information about distance based graph invariants. For instance, knowing the Hosoya
polynomial of a graph, it is straightforward to determine the  Wiener index
of a graph as the first derivative of the polynomial at the point $t=1$. Cash \cite{Cash2002893} noticed that
the hyper-Wiener index can be obtained from the Hosoya polynomial in a similar simple
manner. Also, Estrada et al. \cite{EstradaGutman1998} studied several chemical applications of the
Hosoya polynomial.

Let $G$ be a connected graph of diameter $D$ and let $d(G, k)$, $k \ge  0$, be the number of vertex
pairs at distance $k$. The \textit{Hosoya polynomial} of $G$ is defined as
$$H(G, t) :=\sum_{k=1}^D
d(G, k)\cdot  t^k. $$
As we pointed out above, the Wiener index
of a graph $G$ is determined as the first derivative of the polynomial $ H(G, t)$ at $t=1$, \textit{i.e}., 
$$W(G)=\sum_{k=1}^D
 k \cdot d(G, k). $$

The Hosoya polynomial has been obtained for trees, composite graphs,
benzenoid graphs, tori, zig-zag open-ended nanotubes, certain graph
decorations, armchair open-ended nanotubes, zigzag polyhex nanotorus,
$TUC_4C_8(S)$ nanotubes, pentachains, polyphenyl chains, the
circumcoronene series, Fibonacci and Lucas cubes, Hanoi graphs, etc.
See the references in \cite{Deutsch-Klavzar2013}.

In this note we obtain an explicit formula for the Hosoya polynomial of any distance-regular  graph. As a consequence, we obtain a very simple formula for the Hosoya polynomial of any strongly regular graph.

\section{The Hosoya polynomial of distance-regular graphs}

A \textit{distance-regular}  graph is  a regular connected
 graph with  diameter $D$, for which the following holds. There
are natural numbers $b_0,b_1,...,b_{D-1}$, $c_1=1$,
$c_2,...,c_{D}$ such that for each pair $(u,v)$ of vertices
satisfying $d (u,v)=j$ we have

\begin{itemize}
\item[(1)] the number of vertices in $G_{j-1}(v)$ adjacent to
$u$ is $c_j$  $(1\le j \le D)$,
\item[(2)] the number of vertices in $G_{j+1}(v)$ adjacent to
$u$ is $b_j$  $(0\le j \le D-1),$
\end{itemize}
where $G_i(v)=\{u\in V(G) : d(u,v)=i\}$.

The array $\{b_0,b_1,...,b_{D-1};c_1=1,c_2,...,c_{D}\} $ is the \textit{intersection array} of $G$.


Classes of distance-regular graphs include complete graphs, cycle graphs, Hadamard graphs, hypercube graphs, Kneser graphs $K(n,2)$, odd graphs  and Platonic graphs \cite{Biggs1974,Brouwer1989}.

\begin{theorem} \label{HosoyaPolynomialDist-Reg}
Let $G$ be a  distance-regular  graph whose
intersection array is
$$\{b_0,b_1,...,b_{D-1};c_1=1,c_2,...,c_{D}\}.$$ Then we have
\[
H(G,t)=\frac{nb_0}{2}\left(
t+\sum_{i=2}^{D}\frac{\prod_{j=1}^{i-1}b_j}{\prod_{j=2}^{i}c_j} \cdot t^i
\right).
\]
\end{theorem}

\begin{proof}
For any vertex $v\in V(G)$, each vertex of $G_{i-1}(v)$ is joined to $b_{i-1}$ vertices in $G_{i}(v)$
and each vertex of $G_{i}(v)$ is joined to $c_i$ vertices
in $G_{i-1}(v)$. Thus
\begin{equation} \label{numbertex}
\mid G_{i-1}(v)\mid b_{i-1}=\mid G_{i}(v) \mid c_i.
\end{equation}
Hence, it follows from (\ref{numbertex}) that the number of
vertices at distance $i$ of a vertex $v$, namely $|G_i(v)|$, is
obtained directly from the intersection array
\begin{equation} \label{vertexdist}
|G_i(v)|=\frac{\prod_{j=0}^{i-1}b_j}{\prod_{j=2}^{i}c_j}
 \quad (2\le i \le D) \quad \mbox{\rm and} \quad
 |G_1(v)|=b_0.
 \end{equation}
 Now, since, $d(G,i)=\frac{1}{2}\displaystyle\sum_{v\in V(G)}|G_1(v)|$ and the value $|G_1(v)|$ does not depend on $v$, we obtain the following:
 \begin{equation} \label{vertexdist}
d(G,i)=\frac{n\prod_{j=0}^{i-1}b_j}{2\prod_{j=2}^{i}c_j}
 \quad (2\le i \le D) \quad \mbox{\rm and} \quad
 |G_1(v)|=\frac{nb_0}{2}.
 \end{equation}
 Therefore, the result is a direct consequence of the
 definition of the Hosoya polynomial.
\end{proof}

As an example, the hypercubes  $Q_k$, $k\ge 2$, are distance-regular graphs whose intersection array is
$\{k,k-1,...,1;1,2,...,k\} $, \cite{Biggs1974}. Thus, from Theorem
\ref{HosoyaPolynomialDist-Reg} we obtain that the Hosoya polynomial of the
hypercube $Q_k$ is

$$H(Q_k,t)=2^{k-1}\sum_{i=1}^k \binom{k}{i}t^i=2^{k-1}\left((t+1)^k-1\right).$$

As a direct consequence of Theorem \ref{HosoyaPolynomialDist-Reg} we deduce the  formula on the Wiener index of a distance-regular  graph, which was previously obtained in \cite{Rodriguez2005} for the general case of hypergraphs.

\begin{corollary}{\rm \cite{Rodriguez2005}} \label{wienerdist-reg}
Let $G$ be a  distance-regular  graph whose
intersection array is
$$\{b_0,b_1,...,b_{D-1};c_1=1,c_2,...,c_{D}\}.$$ Then we have
\[
W(G)=\frac{nb_0}{2}\left(
1+\sum_{i=2}^{D}i\frac{\prod_{j=1}^{i-1}b_j}{\prod_{j=2}^{i}c_j}
\right).
\]
\end{corollary}
 A graph is said to be $k$-\textit{regular}   if all vertices have the same degree $k$. A   $k$-regular graph $G$ of order $n$ is said to be \textit{strongly regular}, with parameters $(n,k,\lambda,\mu)$, if the following conditions hold. Each pair of adjacent vertices has the same number $\lambda\ge 0$ of common neighbours, and each pair of non-adjacent vertices has the same number $\mu\ge 1$ of common neighbours (see, for instance, \cite{Biggs1974}). A distance-regular graph  of diameter $D=2$ is simply a strongly regular graph. In terms of the intersection array $\{b_0,b_1;1,c_2\}$ we have that $\lambda=k-1-b_1$ and $\mu=c_2$, \textit{i.e}., the intersection array of any strongly regular graph with parameters $(n,k,\lambda,\mu)$ is $\{k,k-\lambda-1;1,\mu\}$.  
Thus, as a consequence of Theorem \ref{HosoyaPolynomialDist-Reg} we deduce the following result.

\begin{corollary} \label{HosoyaPolynomialStronglyRegGraph}
Let $G$ be a  strongly regular graph with parameters $(n,k,\lambda,\mu)$.  Then we have
\[
H(G,t)=\frac{nk}{2}\left(
t+\frac{k-\lambda-1}{\mu} \cdot t^2
\right).
\]
\end{corollary}

It is well-known that the parameters $(n,k,\lambda,\mu)$ of any strongly regular graph      are not independent and must obey the following relation:

    $$(v-k-1)\mu = k(k-\lambda-1).$$
 As a result, we can express the Hosoya polynomial of any strongly regular graph in the following  manner 
 \[
H(G,t)=\frac{n}{2} 
\left( kt+(n-k-1) t^2\right);
\]
this is not surprising because for every vertex $x$  
there are $k$ vertices at distance 1 from $x$ and $n-k-1$ at distance 2 (since a strongly
regular graph has diameter $D=2$).


\end{document}